\documentclass[
final
 , nomarks
]{dmtcs-episciences}

\usepackage[utf8]{inputenc}
\usepackage{subfigure}

\usepackage{amsmath}

\usepackage[round]{natbib}

\newtheorem{prop}{Proposition}
\newtheorem{lem}[prop]{Lemma}

\newtheorem{thm}[prop]{Theorem}
\newtheorem{cor}[prop]{Corollary}
\newtheorem{prob}[prop]{Problem}

\newcommand{\IN}{\ensuremath{\naturals
}}

\newcommand{\col}{\ensuremath{{\rm col}}}
\newcommand{\dist}{\ensuremath{{\rm dist}}}

\author{Stephan Dominique Andres \and Winfried Hochst\"{a}ttler}
\title[Game colouring number of powers of forests]{The 
game colouring number of powers of forests}
\affiliation{
Fakult\"{a}t  f\"{u}r Mathematik und Informatik, Fernuniversit\"{a}t in
Hagen, Germany}
\received{2015-05-22}
\revised{2015-11-12}
\accepted{2015-11-18}
\begin{document}
\publicationdetails{18}{2015}{1}{2}{648}
\maketitle
\begin{abstract}
We prove that the game colouring number of the $m$-th power of a forest
with maximum degree $\Delta\ge3$ is bounded from
above by
\[\frac{(\Delta-1)^m-1}{\Delta-2}+2^m+1,\]
which improves the best known bound by an asymptotic factor of 2.
\end{abstract}

\section{Introduction}\label{intro}

Graph colouring considers the problem to assign colours to the vertices of a given graph in such a way that adjacent vertices receive 
distinct colours. Classical graph colouring can be regarded as a one-player game, where the single player has the goal to colour every 
vertex in such a way that she uses a minimum number of colours. 
Competitive graph colouring considers the situation that there is a second player, too, who has the goal to increase the number of 
colours used. In a \emph{maker-breaker graph colouring game}, the players are usually called Alice, who tries to minimize the number of 
colours, and Bob, who tries to maximize the number of colours. In the basic variant of such a game, popularised to the graph theory 
community by \cite{bodlaender}, the players move alternately. In each move they colour exactly one uncoloured vertex of the 
given graph the vertices of which are initially uncoloured. Alice 
begins. The  most important parameter considered concerning this game is the so-called \emph{game chromatic number}, which 
is the 
smallest number of colours that is sufficent to colour every vertex in case 
both players use optimal strategies.

The maximum game chromatic number for graphs from many interesting classes of graphs has been examined by many authors. The 
first class of graphs whose maximum game chromatic number was determined were forests. The result is contained in the initial 
paper of \cite{faigleetal}. In order to prove that 4 is an upper bound for the game chromatic number of a tree, Faigle et 
al.\ used a so-called \emph{activation strategy} for Alice. This type of strategy was named and generalized 
by \cite{kierstead} and modified and used by many authors 
to obtain upper 
bounds 
for the game chromatic number of other classes of more complex graphs (some 
references can be found in the survey paper of \cite{bartnickietal}, some more recent references 
concerning graph colouring games can be taken from \cite{andreschargameperf} resp.\ \cite{yang}).  
A remarkable fact about the activation strategy 
is that it does 
not consider the colours of vertices but only the order in which they are coloured.
This fact motivated \cite{zhuplanar} to introduce the following \emph{maker-breaker marking game} defining a graph 
parameter 
that is simultanously an upper bound for the game chromatic number and a competitive version of the colouring number named by 
\cite{erdoeshajnal}. The following rules of this marking game are very similar to those of Bodlaender's 
colouring game.

Alice and Bob alternately mark vertices of a given graph $G=(V,E)$ until every vertex is marked. The way they choose the vertices to be 
marked creates a linear ordering $\le$ on the set $V$, where the smallest element is the first vertex that was marked and the largest 
the last vertex marked. The \emph{back degree} $bd_{\le}(v)$ of a vertex $v$ with respect to the ordering $\le$ is defined as the 
number of previously marked neighbours of~$v$, i.e.
\[bd_{\le}(v):=|\{w\in V\mid vw\in E,w\le v\}|.\]
The \emph{score} $sc(G,\le)$ of $G$ with respect to the linear ordering $\le$ is defined by
\[sc(G,\le):=1+\max_{v\in V}bd_{\le}(v).\]
Alice's goal is to minimize the score, Bob tries to maximize it. Let $\le^\ast$ be a linear ordering in case both players play 
according to optimal strategies. Then the \emph{game colouring number} $\col_g(G)$ of $G$ is defined as
\[\col_g(G):=sc(G,\le^\ast).\]
For a non-empty class ${\mathcal{C}}$ of graphs we define
\[\col_g({\mathcal{C}}):=\sup_{G\in{\mathcal{C}}}\col_g(G).\]

Note that, for any graph,
its game colouring number is greater or equal than its game chromatic number.
An application of the game colouring number with regard to the graph packing problem was given by \cite{kiersteadkostochka}.

In this paper we consider the game colouring number of the class of powers of forests.

We only consider finite, simple, and loopless graphs. By $\binom{V}{2}$ we denote 
the set of 2-element subsets of a set~$V$.
The $m$-th \emph{power} $G^m$ of a graph $G=(V,E)$ is defined as the graph $(V,E_m)$ with 
\[E_m=\{vw\in{\textstyle\binom{V}{2}}\mid 1\le\dist_G(v,w)\le m\},\]
where the \emph{distance} $\dist_G(v,w)$ denotes, as usual, the number of edges on a shortest path from $v$ to~$w$ in~$G$.
In particular, we have $G^0=(V,\emptyset)$ and $G^1=G$.
The \emph{square} of $G$ is the 2nd power $G^2$.

In order to examine the marking game on the power $F^m$ of a forest $F$ we will often argue with the forest $F$ itself, 
which has the same vertex set as $F^m$. The vertex sets are identified in a canonical manner. 
It is useful to define a \emph{$k$-neighbour} of a vertex $v$ as a vertex $w$ with dist$_F(v,w)=k$. 
A \emph{$k_\le$-neighbour} of a vertex $v$ is an $\ell$-neighbour for some $1\le\ell\le k$. 
Hence adjacency in $F^m$ corresponds to $m_{\le}$-neighbourhood in~$F$.

\cite{esperetzhu} and \cite{yang}  determined upper bounds for the game colouring number of squares of graphs 
depending on the 
 maximum degree of the original graph. \cite{andrestheuser} generalized a global bound for squares of graphs from 
the paper of \cite{esperetzhu} to arbitrary powers of graphs and obtained the following upper bound in the special case of 
powers 
of forests.

\begin{thm}[\cite{andrestheuser}]\label{baumsatz} 
Let $F$ be a forest with maximum degree $\Delta\ge3$. Let $m\in{\mathbb N}$. Then we have
\[{\rm col}_g(F^m)\le2\frac{(\Delta-1)^m-1}{\Delta-2}+2.\]
\end{thm}

\vfill

Here we will prove a bound which is better by factor $\approx2$ for large $\Delta$.

\begin{thm}\label{baumsatzneu} 
Let $F$ be a forest with maximum degree $\Delta\ge3$. Let $m\in{\mathbb N}$. Then we have
\[{\rm col}_g(F^m)\le\frac{(\Delta-1)^m-1}{\Delta-2}+2^m+1.\]
\end{thm}

\vfill

\section{Proof of Theorem~\ref{baumsatzneu}}\label{forestsec}

In case $m=1$, Theorem~\ref{baumsatzneu} specializes to the result of \cite{faigleetal} that the game colouring number of 
a forest is at most 4. 

Let us give a brief review of the strategy for Alice Faigle et al.\ essentially used in order to prove this upper bound.
Let $F$ be a forest (with maximum degree~$\Delta$). During the game, a special set $A$ of vertices, called \emph{active} vertices is 
updated. At the beginning, 
$A=\emptyset$. Whenever a 
player marks the first vertex in a component $T$ (which is a tree) of $F$, this vertex is activated and becomes 
root of the 
\emph{tree 
of active vertices} of $T$, which is a rooted tree induced by the vertex set $V(T)\cap A$. We denote the tree of active vertices of 
the 
component~$T$ by $T^A$ and its root by $r(T^A)$. In her first move, Alice marks an arbitrary vertex.
Whenever Bob marks a vertex $v$ in a component $T$, let $w$ be the first active vertex on the path from $v$ to $r(T^A)$ ($v=w$ might be
possible). After Bob's move, every vertex on the path from $v$ to 
$r(T^A)$ is \emph{activated}, i.e.\ it becomes a member of $A$. 
Alice's next move depends on whether $v=w$ or $v\neq w$. 
Note that $v=w$ if and only if $v=r(T^A)$ or $v$ was active ($v\in T^A$) at the time $v$ was marked by Bob.
Alice uses the following strategy:

\bigskip
\vfill

\noindent
\textbf{Alice's basic activation strategy:}   

\begin{description}
\item[Rule A1] If $v\neq w$ and $w$ is unmarked, then Alice marks $w$.
\item[Rule B] Otherwise, Alice chooses a component tree $T_0$ that contains an unmarked vertex and, if $r(T_0^A)$ exists, she marks an 
unmarked vertex with 
smallest 
distance from $r(T_0^A)$, if $r(T_0^A)$ does not exist, she marks a vertex in $T_0$ (which will become $r(T_0^A)$).
\end{description}


\vfill

It is easy to see that if Alice uses this strategy, during the whole game every unmarked vertex has at most two active children. Therefore it has at most 
three marked (1-)neighbours, hence $\col_g(F)\le4$.

\cite{andrestheuser} applied this strategy to the underlying forest $F$ of its $m$-th power $F^m$ in 
order to 
prove Theorem~\ref{baumsatz}. For this purpose we consider the game on $F$ instead of $F^m$ and have to count the maximal number of 
marked $m_{\le}$-neighbours an unmarked vertex may have.

In the proof of Theorem~\ref{baumsatzneu} we use a modification of Alice's basic activation strategy. The main 
difference 
is the additional rule A2, which gives us a significant improvement in the upper bound we establish.


\pagebreak[4]

\noindent
\textbf{Alice's refined activation strategy:}

\begin{description}
\item[Rule A1] If $v\neq w$ and $w$ is unmarked, then Alice marks $w$.
\item[Rule A2] If $v\neq w$ and $w$ is marked and there is an unmarked vertex on the path from $w$ to $r(T^A)$, then Alice marks the 
first 
unmarked vertex on this path (i.e.\ the unmarked vertex on the path that is nearest to~$w$).
\item[Rule B] Otherwise, Alice chooses a component tree $T_0$ that contains an unmarked vertex and, if $r(T_0^A)$ exists, she marks an 
unmarked vertex with 
smallest 
distance from $r(T_0^A)$, if $r(T_0^A)$ does not exist, she marks a vertex in $T_0$ (which will become $r(T_0^A)$).
\end{description}

\begin{proof}[of Theorem~\ref{baumsatzneu}]
Let $m\ge2$. Alice uses the strategy explained above. In the following arguments we consider the underlying forest $F$.
We will show that at any time in the game after Alice's move the invariant holds that any unmarked vertex of~$F$ has at most 
\[M_m:=\frac{(\Delta-1)^m-1}{\Delta-2}+2^m-1\]
marked $m_{\le}$-neighbours. 
Since Bob can increase the number of marked $m_{\le}$-neighbours of an unmarked vertex in his 
next move by at most one,
this means that Alice 
can force a score of at most $M_m+2$
in the marking game on the graph $F^m$. 

We prove the validity of the invariant by induction on the number of moves. In Alice's first move the invariant obviously holds. Assume 
now it holds after some move of Alice. We consider the next pair of moves of Bob and Alice.

We use the same notions, namely the set of active vertices $A$, active rooted tree $T^A$ with root $r(T^A)$ as in the description of the 
special case 
$m=1$ above. To be able to argue more precisely we also consider $T$ as rooted tree with root~$r(T^A)$. In this rooted tree, for a vertex 
$x$, let $p(x)$ be the predecessor of $x$ and $C(x)$ the set of children of $x$.
For $k\ge1$ we define the iterates
\begin{eqnarray*}
p^1(x)&:=&p(x),\\
p^{k+1}(x)&:=&p(p^k(x)),\\
C^0(x)&:=&\{x\},\\
C^{k+1}(x)&:=&
\bigcup_{y\in C^k(x)}C(y).
\end{eqnarray*}
For a vertex $x$, let $c_1,c_2,c_3,\ldots,c_{{\rm deg}(x)-1}$ be the children of $x$ in the order they are 
activated in the course of the game; then we call $c_i$ the \emph{$i$-th active child} of $x$. In 
the following lemmata we use the 
notion of $i$-th active child even before the move after which it is activated.

The rules of the above refined activation strategy imply

\begin{lem}[Consequence of Rule A1]\label{hilfhilfxxy}
At the time Bob marks a vertex in the subtree rooted in the second active child of a vertex $x$, the vertex $x$ will be marked after Alice's move.
\end{lem}

\begin{proof}
We may assume that $x$ is unmarked before Bob's move.
By Rule B, Alice will never mark an inactive vertex that 
is not adjacent to a marked 
vertex. 
Therefore, when Bob, for the first time, marks a vertex $v_B$ in the subtree rooted in the second active 
child $c_2$ of $x$, the vertex $v_B$ is the first active 
vertex in the subtree rooted in $c_2$. Therefore the first active vertex on the path from $v_B$ to the root 
$r(T^A)$ is the vertex~$x$. 
By Rule A1, the vertex $x$ will be marked in Alice's next move.
\end{proof}

By contraposition, we conclude

\begin{lem}\label{hilfsclaimneu}
After Alice's move, for any unmarked vertex $u$, there is at most one child $c\in C(u)$ of $u$ such that in the rooted 
subtree of $c$ (including $c$) there exists at least one marked vertex.\hfill$\Box$
\end{lem}

\begin{lem}[Consequence of Rule A2]\label{consAz}
 At the time Bob marks a vertex in the subtree rooted in the $k$-th active child of a vertex $x$, $k\ge3$, the 
vertices $p(x),\ldots,p^{k-2}(x)$ 
will be marked 
after Alice's move.
\end{lem}

\begin{proof}
As above, by the rules of the game, Alice will never mark an inactive vertex that 
is not adjacent to a marked 
vertex. Therefore the first marked vertex $v_i$ in the $i$-th child tree of $x$, $i=2,\ldots,k$, must be marked by Bob. 
By Lemma~\ref{hilfhilfxxy}, the 
vertex $x$ will be marked after Alice's move immediately after Bob marked~$v_2$. 
By induction on $i$, it follows from Rule A2 that $p^{i-2}(x)$ will be marked after Alice's move immediately 
after Bob marked $v_i$, $i=3,\ldots,k$.
\end{proof}

Let $u$ be an unmarked vertex after Alice's move.
By Lemma~\ref{hilfsclaimneu}, at most two neighbours of $u$ are marked, one child $c_0$ and the parent $p(u)$.
We will determine  
\begin{itemize}
\item[(i)] an
upper bound for the number of vertices in $V(T^A)\cap\bigcup_{k=1}^{m}C^k(u)$, namely 
\[2^m-1,\]
and
\item[(ii)] an upper bound for the number of vertices in the ancestor's part of the active tree with distance at most $m$ from $u$, 
namely 
\[\frac{(\Delta-1)^m-1}{\Delta-2}.\]
\end{itemize}
 Summing these two values obviously gives an upper bound for the number of marked $m_{\le}$-neighbours of an 
unmarked vertex after Alice's 
move. 

\bigskip

\emph{Bound in (i):} This bound is proved by a series of lemmata. We first introduce a key notion. 
A \emph{big vertex} is a
vertex $z\in V(T^A)\cap\bigcup_{k=1}^{m-1}C^k(u)$ with the property
\begin{itemize}
\item[(V1)] either $z$ has $b\ge3$ active children and was marked by Alice by Rule A1,
\item[(V2)] or $z$ has $b\ge2$ active children and was marked by Alice by Rule A2 or marked by Bob.
\end{itemize}
A \emph{rabbit} of a big vertex $z$ is an active child $c$ of $z$ which is in case (V1) neither the first nor the second active 
child of $z$ and in case (V2) not the first active child of $z$.

Let $S_1$ resp.\ $S_2$ be the set of rabbits of some big vertex of type (V1) resp.\ type (V2). 
Let $B_2$ be the set of big vertices of type (V2).
Let $D_2$ be the set of active vertices in 
$\bigcup_{k=1}^{m-2}C^k(u)$ with exactly one active child. 

Arguing with Rule A2 we can prove

\begin{lem}\label{haupthilf}
\begin{itemize}
\item[(a)]
 Let $z\in C^{k}(u)$, $1\le k\le m-1$, be a big vertex of type (V1) with $b$ children. Then there exist $b-2$ vertices from 
$V(T^A)\cap\bigcup_{i=1}^{k-1}C^i(u)$
which were marked by Alice by Rule A2 when Bob marked a vertex in the subtree rooted in a rabbit of~$z$.
\item[(b)]
 Let $z\in C^{k}(u)$, $1\le k\le m-1$, be a big vertex of type (V2) with $b$ children. Then there exist $b-1$ vertices from 
$V(T^A)\cap\bigcup_{i=1}^{k-1}C^i(u)$
which were marked by Alice by Rule A2 when Bob marked a vertex in the subtree rooted in a rabbit of~$z$.
\end{itemize}
\end{lem}

\begin{proof}
Alice, by her strategy, will never mark a vertex in a subtree rooted in a rabbit of~$z$ unless every vertex on the path 
from 
$z$ to $r(T^A)$ is marked (which does not hold since $u=p^k(z)$ is unmarked). Therefore, every rabbit will be created by Bob. Moreover 
in case~(a), by Lemma~\ref{hilfhilfxxy}, $z$ will be marked by Alice before Bob creates the first rabbit. In case (b), again by 
Lemma~\ref{hilfhilfxxy}, $z$ will be 
marked by Bob or by Alice before Bob creates the first rabbit, otherwise $z$ would not be of type~(V2). Since $z$ is marked, the path 
from $z$ to $u$ is active before the first rabbit is created. Whenever Bob creates a rabbit, Alice marks a vertex on the path from 
$z$ to $r(T^A)$. Since $u$ is unmarked, this vertex, by Lemma~\ref{consAz}, indeed must lie on the path from $z$ to~$u$, i.e.\ the vertex is in 
$V(T^A)\cap\bigcup_{i=1}^{k-1}C^i(u)$, which proves the lemma.
\end{proof}

The preceding lemma helps us to prove the following key lemma of the proof.

\begin{lem}
There exists an injective mapping 
$f:\;S_1\cup S_2\longrightarrow B_2\cup D_2$.
\end{lem}

\begin{proof}
The mapping $f$ is defined by Alice's reaction on Bob's moves creating a rabbit of some big vertex: Each such rabbit is mapped to the 
vertex Alice marks immediately in the next move. By construction and the rules of the game, the mapping is injective. We only have to 
show that it is well-defined, i.e.\ that it maps to $B_2\cup D_2$. According to the proof of Lemma~\ref{haupthilf} a vertex in the 
image $f(S_1\cup S_2)$ lies in the set $\bigcup_{k=1}^{m-2}C^k(u)$.
It suffices to show that such a vertex $y\in f(S_1\cup S_2)$ will never become a big vertex of type (V1). But this follows from the 
fact 
that $y$ is marked by Alice by rule A2, implying that $y$ has at most one active child. Therefore $y$ will be in $D_2$ as long as it 
has still one active child and become a big vertex of type (V2), i.e.\ a member of $B_2$, whenever a second child is activated.
\end{proof}

In order to analyse the number of marked neighbours of $u$, a vertex which is unmarked at the current state of the game, we modify the part $T'$ of the active 
tree $T^A$ with vertex set 
$V(T'):=V(T^A)\cap\bigcup_{k=1}^{m}C^k(u)$ in the following way. For every big vertex $z$ and every rabbit $c$ of $z$ we delete the 
active 
subtree 
rooted in $c$, move it, and append it as a child of $f(c)$.

\begin{lem}\label{binarylem}
In the modified tree $T_0$ of the part $T'$ of the active tree every  vertex has at most two children, i.e.\ $T_0$ is a binary tree 
rooted 
in the unique 
active child of~$u$.
\end{lem}

\begin{proof}
\begin{enumerate}[{Case }1:]
\item
Assume that $v=f(c)$ is a vertex that receives new children after the modification.
If $f(c)\in D_2$, the vertex $f(c)$ has exactly  one active child before the modification and gets exactly another one after the 
modification. If $f(c)\in 
B_2$, it has 
exactly one active non-rabbit child before the modifcation. Since every subtree rooted in a rabbit of $f(c)$ is deleted, after the 
modification $f(c)$ has exactly two active children. 

\item\label{casetwo}
Assume that $v$ is a vertex that does not receive marked children during the modification. In case $v$ is a big vertex of type (V1), after 
the modification all but two active child subtrees have been moved away. If $v$ is a big vertex of type (V2), under the assumption of Case~\ref{casetwo},
after the modification all but one active child subtree has been moved away. If $v$ is not a big vertex, then, by the definition of big vertices, Alice's 
strategy and the assumption of 
Case~\ref{casetwo}, $v$ has at most two active children.
\end{enumerate}
This proves the lemma.
\end{proof}

\begin{lem}\label{binarybound}
$T_0$ has the same number of vertices as the original part $T'$ of the active tree and none  of the moved vertices lies outside 
$\bigcup_{k=1}^{m}C^k(u)$.
\end{lem}

\begin{proof}
The first assertion follows since we do not delete subtrees, moreover, we move them. The second follows from the fact that we move them 
to a lower level in the binary tree, but not below the level of the root (the unique active child of $u$).
\end{proof}

\begin{cor}\label{endcor}
$T_0$ has at most $2^m-1$ vertices.
\end{cor}

\begin{proof}
By Lemma~\ref{binarylem}, $T_0$ is a binary tree. By the second assertion of Lemma~\ref{binarybound}, $T_0$ has height at most $m-1$. 
Therefore, at level $i$ we 
have at most $2^i$ vertices, $i=0,\ldots,m-1$. 
Summing up, we get at most
\[\sum_{i=0}^{m-1}2^i=2^m-1\]
vertices.
\end{proof}

\begin{cor}\label{endendx}
The number of marked $m_{\le}$-neighbours of $u$ in the child trees of $u$ is at most
\begin{equation}\label{vgli}
\left|V(T^A)\cap\bigcup_{k=1}^{m}C^k(u)\right|\le 2^m-1.
\end{equation}
\end{cor}

\begin{proof}
By the first assertion of Lemma~\ref{binarybound} and Corollary~\ref{endcor}
\[\left|V(T^A)\cap\bigcup_{k=1}^{m}C^k(u)\right|=|V(T_0)|\le 2^m-1.\]
\end{proof}

Corollary~\ref{endendx} gives us the desired bound (i).

\bigskip

\emph{Bound in (ii):}
If we consider $p(u)$ and consider the tree $T$ as rooted in $u$, then in the new ``child'' subtree rooted in $p(u)$
(which is the tree of foremothers and aunts and so on)
there might be at most
$(\Delta-1)^{k-1}$ marked vertices at distance $k$ from $u$ in the new subtree.
Therefore the number of marked vertices in the new subtree is at most
\begin{equation}\label{vglvor}
\sum_{k=0}^{m-1}(\Delta-1)^k=\frac{(\Delta-1)^m-1}{\Delta-2}.
\end{equation}

\bigskip

Combining (i) and (ii), i.e.\ adding the bounds (\ref{vgli}) and (\ref{vglvor}), in total the number of $m_{\le}$-neighbours after Alice's move is at 
most
\[\frac{(\Delta-1)^m-1}{\Delta-2}+2^m-1=M_m.\]
This completes the proof of Theorem~\ref{baumsatzneu}.
\end{proof}

\section{Open problems}

\cite{andrestheuser} specify a lower bound for the game colouring number of the class of \mbox{$m$-th} powers of forests with maximum degree $\Delta$, 
based 
on an observation of \cite{agnahall}, which is 
$\Omega(\Delta^{\left\lfloor\frac{m}{2}\right\rfloor})$. Therefore even the improved bound in Theorem~\ref{baumsatzneu} leaves a large 
asymptotic gap between lower and upper bound.

\begin{prob}\label{probeinsxxx}
Let ${\mathcal{F}}$ be the class of forests with maximum degree $\Delta$ and $m\in\IN$. Determine 
\[\col_g(\{F^m\mid 
F\in{\mathcal{F}}\}).\]
\end{prob}

If $m=2$ and $\Delta\ge9$, the gap in Problem~\ref{probeinsxxx} was reduced by \cite{esperetzhu} who proved that 
\[\Delta+1\le\col_g\{F^2\mid F\in{\mathcal{F}}\}\le\Delta+3.\]

It might be that a generalization of the activation strategy can be applied to powers of members of graph classes with some tree 
decomposition structure.

\begin{prob}
Let ${\mathcal{T}}_k$ be the class of partial $k$-trees with maximum degree $\Delta$ and $m\in\IN$. Determine 
\[\col_g(\{G^m\mid 
G\in{\mathcal{T}}_k\}).\]
\end{prob}

More generally,

\begin{prob}
Let ${\mathcal{G}}_k$ be the class of $k$-degenerate graphs with maximum degree $\Delta$ and $m\in\IN$. Determine 
\[\col_g(\{G^m\mid 
G\in{\mathcal{G}}_k\}).\]
\end{prob}

Exact values for the game colouring number of powers of special forests are only known for large paths, cf.\ \cite{andrestheuser}.

\begin{prob}
Determine the exact values $\col_g(F^m)$ for all $m\in\IN$ and interesting special forests $F$.
\end{prob}

\nocite{*}
\bibliographystyle{abbrvnat}
\bibliography{final_Andres_Hochstaettler_colpowerforest}
\label{sec:biblio}

\end{document}